\documentclass[12pt,reqno]{article}

\usepackage[usenames]{color}
\usepackage{amssymb}
\usepackage{amsmath}
\usepackage{amsthm}
\usepackage{amsfonts}
\usepackage{amscd}
\usepackage{graphicx}

\usepackage[colorlinks=true,
linkcolor=webgreen,
filecolor=webbrown,
citecolor=webgreen]{hyperref}

\definecolor{webgreen}{rgb}{0,.5,0}
\definecolor{webbrown}{rgb}{.6,0,0}

\usepackage{color}
\usepackage{fullpage}
\usepackage{float}

\usepackage{graphics}
\usepackage{latexsym}

\newcommand{\braces}{\genfrac{\lbrace}{\rbrace}{0pt}{}}

\begin{document}

\theoremstyle{plain}
\newtheorem{theorem}{Theorem}
\newtheorem{corollary}[theorem]{Corollary}
\newtheorem{proposition}{Proposition}
\newtheorem{lemma}{Lemma}
\newtheorem{example}{Examples}
\newtheorem{remark}{Remark}

\begin{center}
\vskip 1cm
{\LARGE\bf
Three combinatorial sums involving central binomial coefficients
}

\vskip 1cm

{\large
Kunle Adegoke \\
Department of Physics and Engineering Physics, \\ Obafemi Awolowo University, 220005 Ile-Ife, Nigeria \\
\href{mailto:adegoke00@gmail.com}{\tt adegoke00@gmail.com}

\vskip 0.2 in

Robert Frontczak \\
Independent Researcher, 72762 Reutlingen \\ Germany \\
\href{mailto:robert.frontczak@web.de}{\tt robert.frontczak@web.de}
}

\end{center}

\vskip .2 in

\begin{abstract}
We study three classes of combinatorial sums involving central binomial coefficients and harmonic numbers, 
odd harmonic numbers, and even indexed harmonic numbers, respectively. In each case we use summation by parts to derive recursive
expressions for these sums. In addition, we offer an alternative approach to express one class of sums and some related sums 
in closed form in terms of Stirling numbers and $r$-Stirling numbers of the second kind. 
\end{abstract}



\bigskip

\section{Introduction}

The central binomial coefficients are defined for $n=0,1, \ldots,$ by
\begin{equation*}
\binom {2n}{n} = \frac{(2n)!}{(n!)^{2}}.
\end{equation*}
These numbers are listed under the id A000984 in the OEIS database \cite{OEIS}. They are closely related to the famous Catalan numbers
(entry A000108 in \cite{OEIS})
\begin{equation*}
C_{n} =\frac{1}{n+1} \binom{2n}{n}
\end{equation*}
and possess many intriguing properties (\cite{Just,Mattarei,Mikic,Pomerance,Sun}). \\

Riordan \cite[p. 130]{Riordan} proved the combinatorial identity
\begin{equation}\label{Riordan1}
\sum_{k=0}^n 2^{-2k} \binom{2k}{k} \frac{1}{1-2k} = 2^{-2n} \binom{2n}{n}.
\end{equation}
This identity has a dual counterpart as given by
\begin{equation}\label{eq.fw0v51n}
\sum_{k=0}^n (-1)^k \binom{n}{k} \binom{2k}{k}^{-1} 2^{2k} = \frac{1}{1-2n}.
\end{equation}
See \cite{Adegoke2} for a generalization of these sums. Alzer and Nagy \cite{Alzer} proved several identities involving
the central binomial coefficients and Catalan numbers. Boyadzhiev \cite{Boyadzhiev} considered four identities for the squared central binomial coefficients. \\

The following combinatorial identity is also well-known:
\begin{equation}\label{central_bin}
\sum_{k=0}^n 2^{-2k} \binom{2k}{k} = 2^{-2n} (2n+1) \binom{2n}{n}.
\end{equation}
This identity is stated, for instance, as Theorem 23 in \cite{Qi} and as Theorem 4 in \cite{Guo}.
Both proofs are based on integration and apply integral representations for the Catalan numbers.
In \cite[Theorem 7.1]{Bataille} the following generalization was proved: For $r\geq 0$ define
\begin{equation}\label{central_bin_gen}
S_r(n) = \sum_{k=0}^n 2^{-2k} k^r \binom{2k}{k}, \qquad S_0(n) = \sum_{k=0}^n 2^{-2k} \binom{2k}{k}.
\end{equation}
Then
\begin{equation}\label{eq1_theorem7_1}
S_0(n) = 2^{-2n} (2n+1) \binom{2n}{n}
\end{equation}
and
\begin{equation}\label{eq2_theorem7_1}
(2r+1) S_r(n) = n^r 2^{-2n} (2n+1) \binom{2n}{n} + 2\sum_{j=1}^{r-1}\binom{r}{j+1}(-1)^{j+1}S_{r-j}(n).
\end{equation}
For $r=1,2,3$ the recursion \eqref{eq2_theorem7_1} immediately produces the identities
\begin{equation}\label{cb_sum1}
\sum_{k=0}^n 2^{-2k} k \binom{2k}{k} = \frac{n}{3} 2^{-2n} (2n+1) \binom{2n}{n},
\end{equation}
\begin{equation}\label{cb_sum2}
\sum_{k=0}^n 2^{-2k} k^2 \binom{2k}{k} = \frac{n}{15} 2^{-2n} (3n+2) (2n+1) \binom{2n}{n},
\end{equation}
and
\begin{equation}\label{cb_sum3}
\sum_{k=0}^n 2^{-2k} k^3 \binom{2k}{k} = \frac{n}{105} 2^{-2n} (15n^2+18n+2) (2n+1) \binom{2n}{n}.
\end{equation}

As was observed in \cite{Bataille}, for $r\ge 1$, the sum $S_r(n)$ is of the form $n(2n+1)2^{-2n}\binom{2n}{n}P_r(n)$
for some polynomial $P_r(x)$ of degree $m-1$. The recursion \eqref{eq2_theorem7_1} shows that the polynomials $P_r(x)$
satisfy $P_1(x)=\frac{1}{3}$ and
\begin{equation*}
(2r+1) P_r(x) = x^{r-1} + 2 \sum_{j=1}^{r-1} \binom{r}{j+1}(-1)^{j+1} P_{r-j}(x).
\end{equation*}

The next combinatorial sum is very similar to \eqref{central_bin} and involves the familiar odd harmonic numbers $O_n$ (see \cite{Adegoke}):
\begin{equation}\label{cb_ohn_k0}
\sum_{k=0}^n 2^{-2k} \binom{2k}{k} O_k = 2^{-2n} \binom{2n}{n}\left ((2n+1)O_n - 2n \right ).
\end{equation}
Here and throughout the paper $O_n,n\geq 0,$ are the odd harmonic numbers and $H_n,n \geq 0,$ are the ordinary harmonic numbers, i.e.,
\begin{equation*}
H_0 = 0, \quad H_n = \sum_{m=1}^n \frac{1}{m} \qquad \text{and} \qquad O_0 = 0, \quad O_n = \sum_{m=1}^n \frac{1}{2m-1}.
\end{equation*}

In this paper, we are concerned with the following three classes of combinatorial sums defined for all integers $r \geq 0$:
\begin{equation}\label{central_bin_hn}
S_r^H(n) = \sum_{k=0}^n 2^{-2k} k^r \binom{2k}{k} H_k, \qquad S_0^H(n) = \sum_{k=0}^n 2^{-2k} \binom{2k}{k} H_k,
\end{equation}
its odd variant
\begin{equation}\label{central_bin_ohn}
S_r^O(n) = \sum_{k=0}^n 2^{-2k} k^r \binom{2k}{k} O_k, \qquad S_0^O(n) = \sum_{k=0}^n 2^{-2k} \binom{2k}{k} O_k,
\end{equation}
and their even indexed cousin
\begin{equation}\label{central_bin_hn_even}
S_r^{2H}(n) = \sum_{k=0}^n 2^{-2k} k^r \binom{2k}{k} H_{2k}, \qquad S_0^{2H}(n) = \sum_{k=0}^n 2^{-2k} \binom{2k}{k} H_{2k}.
\end{equation}

We will show that there are close connections between the sums $S_r^H(n)$, respectively $S_r^O(n)$ and $S_r^{2H}(n)$, and $S_r(n)$.
More precisely, we will use summation by parts solve these sums recursively for each $r$, hereby using the recursion for $S_r(n)$ 
as in \eqref{eq2_theorem7_1}. In addition, we will offer an alternative approach to express $S_r^O(n)$ and some related sums in closed form 
in terms of Stirling numbers and $r$-Stirling numbers of the second kind.

\section{The sums $S_r^H(n)$ and $S_r^O(n)$}

As will become obvious in a sequel, it is better to start with the sums $S_r^O(n)$. Our first main result in this paper is the next theorem.

\begin{theorem}\label{main_1}
Let $S_r^O(n)$ be defined as in \eqref{central_bin_ohn}. Then
\begin{equation}
S_0^O(n) = 2^{-2n} \binom{2n}{n}\left ((2n+1)O_n - 2n \right )
\end{equation}
and for $r\geq 1$ we have the following recursion:
\begin{equation}\label{main_eq1}
(2r+1) S_r^O(n) = n^r 2^{-2n} \binom{2n}{n} \left ((2n+1)O_n + 1 \right ) - S_r(n) + 2\sum_{j=1}^{r-1}\binom{r}{j+1}(-1)^{j+1}S_{r-j}^O(n),
\end{equation}
where $S_r(n)$ is defined in \eqref{central_bin_gen} and given recursively in \eqref{eq2_theorem7_1}.
\end{theorem}
\begin{proof}
Let us define two sequences $u_k=2^{-2k}\binom{2k}{k} O_k$ and $w_k=2^{-2k}\binom{2k}{k}$. First we note that
\begin{equation}\label{wk_prop}
w_k = 2(k+1)(w_k - w_{k+1}).
\end{equation}
Using this property we have
\begin{align*}
u_k - u_{k+1} &= 2^{-2k}\binom{2k}{k} O_k - 2^{-2(k+1)}\binom{2(k+1)}{k+1} O_k - 2^{-2(k+1)}\binom{2(k+1)}{k+1} \frac{1}{2k+1} \\
&= \frac{1}{2(k+1)} 2^{-2k}\binom{2k}{k} O_k - \frac{1}{2(k+1)} 2^{-2k}\binom{2k}{k}.
\end{align*}
Hence,
\begin{equation*}
2(k+1) (u_k - u_{k+1}) = u_k - w_k = u_k - 2(k+1) (w_k - w_{k+1})
\end{equation*}
or
\begin{equation*}
u_k = 2(k+1) (u_k + w_k - (u_{k+1} + w_{k+1})).
\end{equation*}
From here we can use telescoping and calculate
\begin{align*}
\sum_{k=0}^n u_k &= 2 \left( \sum_{k=0}^n (k+1)(u_k - u_{k+1}) + \sum_{k=0}^n (k+1)(w_k - w_{k+1})\right ) \\
&= 2 \sum_{k=0}^n (u_k + w_k) - 2(n+1)(u_{n+1} + w_{n+1})
\end{align*}
and, using \eqref{central_bin},
\begin{equation*}
\sum_{k=0}^n u_k = - 2 \sum_{k=0}^n w_k + 2(n+1)(u_{n+1} + w_{n+1}) = 2^{-2n} \binom{2n}{n}\left ((2n+1)O_n - 2n \right )
\end{equation*}
as stated in \eqref{cb_ohn_k0}. Now, suppose that $r\ge 1$. To calculate $S_r^O(n)$ we use summation by parts with
$z_k=u_k+w_k$ and $a_k=k(k-1)^r$. We have
\begin{align*}
S_r^O(n) &= 2\sum_{k=0}^n k^r(k+1) (z_k - z_{k+1}) \\
&= 2\sum_{k=0}^n a_{k+1}(z_k - z_{k+1}) \\
&= 2\sum_{k=0}^n z_{k} (a_{k+1} - a_{k}) - 2a_{n+1} z_{n+1} \\
&= 2 \sum_{k=0}^n z_{k} \left ( (r+1)k^r - \sum_{j=2}^r \binom{r}{j}(-1)^j k^{r-j+1}\right) - 2z_{n+1}(n+1) n^r \\
&= (2r+2) \left ( S_r^O(n) + S_r(n) \right ) - 2(n+1)n^r(u_{n+1}+w_{n+1}) \\
& \qquad - 2 \sum_{j=2}^r \binom{r}{j} (-1)^j \left ( S_{r-j+1}^O(n) + S_{r-j+1}(n) \right ).
\end{align*}
Shifting the summation index and simplifying yields
\begin{align*}
(2r+1) S_r^O(n) &= n^r 2^{-2n} \binom{2n}{n}\left ((2n+1)O_n + 2n +2\right ) - (2r+2) S_r(n) \\
&\qquad + 2 \sum_{j=1}^{r-1} \binom{r}{j+1} (-1)^{j+1} \left ( S_{r-j}^O(n) + S_{r-j}(n) \right ).
\end{align*}
The final expression follows by inserting \eqref{eq2_theorem7_1} and again simplifying.
\end{proof}

\begin{corollary}
We have
\begin{equation}\label{cb_ohn_sum1}
\sum_{k=0}^n 2^{-2k} k \binom{2k}{k} O_k = \frac{n}{3} 2^{-2n} \binom{2n}{n}\Big ((2n+1)O_n - \frac{2}{3}(n-1)\Big ),
\end{equation}
\begin{equation}\label{cb_ohn_sum2}
\sum_{k=0}^n 2^{-2k} k^2 \binom{2k}{k} O_k = \frac{n}{15} 2^{-2n} \binom{2n}{n}\Big ((3n+2)(2n+1)O_n - \frac{2}{15}(9n-7)(n-1)\Big ),
\end{equation}
and
\begin{align}\label{cb_ohn_sum3}
\sum_{k=0}^n 2^{-2k} k^3 \binom{2k}{k} O_k &= \frac{n}{105} 2^{-2n} \binom{2n}{n}\Big ((15n^2+18n+2)(2n+1)O_n \nonumber \\
&\qquad\qquad\qquad\qquad - \frac{2}{105}(225n^2+198n-71)(n-1)\Big ).
\end{align}
\end{corollary}
\begin{proof}
Set $r=1,2,3$ in \eqref{main_eq1} and combine with \eqref{cb_sum1}-\eqref{cb_sum3}.
\end{proof}

\begin{remark}
As $O_{n+1}=O_n+\frac{1}{2n+1}$ we can express the special cases equivalently as
\begin{equation}\label{cb_ohn_sum1_mod}
\sum_{k=0}^n 2^{-2k} k \binom{2k}{k} O_k = \frac{n}{3} (2n+1) 2^{-2n} \binom{2n}{n}\Big (O_{n+1} - \frac{1}{3}\Big ),
\end{equation}
\begin{equation}\label{cb_ohn_sum2_mod}
\sum_{k=0}^n 2^{-2k} k^2 \binom{2k}{k} O_k = \frac{n}{15} (2n+1) 2^{-2n} \binom{2n}{n}\Big ((3n+2)O_{n+1} - \frac{1}{15}(9n+16)\Big ),
\end{equation}
and
\begin{align}\label{cb_ohn_sum3_mod}
\sum_{k=0}^n 2^{-2k} k^3 \binom{2k}{k} O_k &= \frac{n}{105} (2n+1) 2^{-2n} \binom{2n}{n}\Big ((15n^2+18n+2)O_{n+1} \nonumber \\
&\qquad\qquad\qquad\qquad - \frac{1}{105}(225n^2+648n+352)\Big ).
\end{align}
\end{remark}

Our second theorem contains the results for the sums $S_r^H(n)$.

\begin{theorem}\label{main_2}
Let $S_r^H(n)$ be defined as in \eqref{central_bin_hn}. Then
\begin{equation}
S_0^H(n) = (2n+1) 2^{-2n} \binom{2n}{n} (H_n - 2) + 2
\end{equation}
and for $r\geq 1$ we have the following recursion:
\begin{align}\label{main_eq2}
(2r+1) S_r^H(n) &= n^r (2n+1) 2^{-2n} \binom{2n}{n} \left (H_n + 1 \right ) + 2 (-1)^r - S_r(n) \nonumber \\
&\qquad + 2\sum_{j=1}^{r-1}\binom{r}{j+1} (-1)^{j+1} S_{r-j}^H(n) - 2\sum_{j=0}^{r}\binom{r+1}{j} (-1)^{r-j} S_{j}(n),
\end{align}
where $S_r(n)$ is defined in \eqref{central_bin_gen} and given recursively in \eqref{eq2_theorem7_1}.
\end{theorem}
\begin{proof}
We apply the same ideas as in the proof of Theorem \ref{main_1}. We start with the two sequences
$v_k=2^{-2k}\binom{2k}{k} H_k$ and $w_k=2^{-2k}\binom{2k}{k}$. Then
\begin{align*}
v_k - v_{k+1} &= 2^{-2k}\binom{2k}{k} H_k - 2^{-2(k+1)}\binom{2(k+1)}{k+1} H_k - 2^{-2(k+1)}\binom{2(k+1)}{k+1} \frac{1}{k+1} \\
&= \frac{1}{2(k+1)} 2^{-2k}\binom{2k}{k} H_k - \frac{1}{2(k+1)} 2^{-2k}\binom{2k}{k}\left (1+\frac{k}{k+1}\right ).
\end{align*}
Hence, using \eqref{wk_prop},
\begin{equation*}
2(k+1) (v_k - v_{k+1}) = v_k - 2(k+1) (w_k - w_{k+1}) - 2k (w_k - w_{k+1})
\end{equation*}
or
\begin{equation*}
v_k = 2(k+1) (v_k - v_{k+1} + w_k - w_{k+1}) + 2k (w_k - w_{k+1}).
\end{equation*}
An application of the telescoping property yields
\begin{equation*}
\sum_{k=0}^n v_k = 2 \sum_{k=0}^n v_k - 2(n+1)v_{n+1} + 4 \sum_{k=0}^n w_k - 4(n+1)w_{n+1} - 2 w_0 + 2 w_{n+1}
\end{equation*}
and finally
\begin{align*}
S_0^H(n) &= \sum_{k=0}^n v_k = 2(n+1)v_{n+1} + (4n+2)w_{n+1} - 4 \sum_{k=0}^n w_k + 2 w_0 \\
&= (2n+1) 2^{-2n} \binom{2n}{n} \left (H_{n+1} + \frac{2n+1}{n+1} - 4 \right ) + 2,
\end{align*}
and the expression for $S_0^H(n)$ follows. Now, suppose that $r\ge 1$. To calculate $S_r^H(n)$ we use summation by parts twice.
We set $Z_k=v_k+w_k$ and $a_k=k(k-1)^r$, and calculate
\begin{align*}
S_r^H(n) &= 2\sum_{k=0}^n k^r(k+1) (Z_k - Z_{k+1}) + 2 \sum_{k=0}^n k^{r+1} (w_k - w_{k+1}) \\
&= 2\sum_{k=0}^n a_{k+1}(Z_k - Z_{k+1}) + 2 \sum_{k=0}^n k^{r+1} (w_k - w_{k+1}) \\
&= 2\sum_{k=0}^n Z_{k} (a_{k+1} - a_{k}) - 2a_{n+1} Z_{n+1} + 2 A_r(n),
\end{align*}
where we have set
\begin{equation*}
A_r(n) = \sum_{k=0}^n k^{r+1} (w_k - w_{k+1}).
\end{equation*}
This gives
\begin{align*}
S_r^H(n) &= (2r+2) \left ( S_r^H(n) + S_r(n) \right ) - n^r (2n+1) 2^{-2n} \binom{2n}{n} (H_{n+1} + 1) + 2A_r(n) \\
& \qquad - 2 \sum_{j=1}^{r-1} \binom{r}{j+1} (-1)^{j+1} \left ( S_{r-j}^H(n) + S_{r-j}(n) \right )
\end{align*}
or
\begin{equation*}
(2r+1) S_r^H(n) = n^r (2n+1) 2^{-2n} \binom{2n}{n} H_{n+1} - S_r(n) - 2A_r(n) + 2 \sum_{j=1}^{r-1} \binom{r}{j+1} (-1)^{j+1} S_{r-j}^H(n).
\end{equation*}
It remains to provide a link between $A_r(n)$ and $S_r(n)$. This is done by applying partial summation a second time.
With $b_k=(k-1)^{r+1}$ we have
\begin{align*}
A_r(n) &= \sum_{k=0}^n b_{k+1} (w_k - w_{k+1}) \\
&= \sum_{k=0}^n w_{k} (b_{k+1} - b_k) - w_{n+1}b_{n+1} + w_0 b_0 \\
&= \sum_{k=0}^n w_{k} \sum_{j=0}^r \binom{r+1}{j} k^j (-1)^{r-j} - w_{n+1}b_{n+1} + w_0 b_0 \\
&= \sum_{j=0}^r \binom{r+1}{j} k^j (-1)^{r-j} S_j(n) - \frac{2n+1}{2(n+1)} 2^{-2n} \binom{2n}{n} n^{r+1} + (-1)^{r+1}
\end{align*}
and the theorem is proved.
\end{proof}

\begin{corollary}
We have
\begin{equation}\label{cb_hn_sum1}
\sum_{k=0}^n 2^{-2k} k \binom{2k}{k} H_k = \frac{1}{3} (2n+1) 2^{-2n} \binom{2n}{n}\Big (n H_n - \frac{2}{3}n + 2\Big ) - \frac{2}{3},
\end{equation}
\begin{equation}\label{cb_hn_sum2}
\sum_{k=0}^n 2^{-2k} k^2 \binom{2k}{k} H_k = \frac{1}{5} (2n+1) 2^{-2n} \binom{2n}{n}\Big (\frac{n}{3}(3n+2)H_n - \frac{2}{45}(9n^2-14n+15)\Big ) + \frac{2}{15},
\end{equation}
and
\begin{align}\label{cb_hn_sum3}
\sum_{k=0}^n 2^{-2k} k^3 \binom{2k}{k} H_k &= \frac{1}{7} (2n+1) 2^{-2n} \binom{2n}{n}\Big (\frac{n}{15}(15n^2+18n+2)H_n \nonumber \\
&\qquad - \frac{2}{1575}(225n^3-297n^2+37n+105)\Big ) + \frac{2}{105}.
\end{align}
\end{corollary}
\begin{proof}
Set $r=1,2,3$ in \eqref{main_eq2} and combine with \eqref{cb_sum1}-\eqref{cb_sum3}.
\end{proof}

\section{Further observations}

From the two main results derived in the previous section we can also deduce the third main result dealing with sums with even indexed harmonic numbers $S_r^{2H}(n)$ given in \eqref{central_bin_hn_even}. For this class of combinatorial sums the following recursion is true.

\begin{theorem}
We have
\begin{equation}
S_0^{2H}(n) = \sum_{k=0}^n 2^{-2k} \binom{2k}{k} H_{2k} = 2^{-2n} \binom{2n}{n} \left ((2n+1)H_{2n} - (4n+1)\right ) + 1,
\end{equation}
and for $r\geq 1$ the following recursion holds:
\begin{align}\label{main_eq3}
(2r+1) S_r^{2H}(n) &= n^r 2^{-2n} \binom{2n}{n} \left ((2n+1)H_{2n} + n + \frac{3}{2}\right ) + (-1)^r - \frac{3}{2} S_r(n) \nonumber \\
&\qquad + 2\sum_{j=1}^{r-1}\binom{r}{j+1} (-1)^{j+1} S_{r-j}^{2H}(n) - \sum_{j=0}^{r}\binom{r+1}{j} (-1)^{r-j} S_{j}(n),
\end{align}
where $S_r(n)$ is defined in \eqref{central_bin_gen} and given recursively in \eqref{eq2_theorem7_1}. In particular,
\begin{equation}
\sum_{k=0}^n 2^{-2k} k \binom{2k}{k} H_{2k} =
\frac{1}{9} 2^{-2n} \binom{2n}{n} \left (3n (2n+1)  H_{2n} - 4n^2 + 7n + 3 \right ) - \frac{1}{3}.
\end{equation}
\end{theorem}
\begin{proof}
The basic connection between $H_n$ and $O_n$ 
$$
O_n = H_{2n} - \frac{1}{2} H_n,
$$
translates immediately to
$$
S_r^{2H}(n) = S_r^{O}(n) + \frac{1}{2} S_r^H(n)
$$
and the recursive expression is deduced from Theorems \ref{main_1} and \ref{main_2}.
\end{proof}

\section{An alternative approach}

In this section, based on a known polynomial identity, we express $S_0^O (n)$ and some related sums in closed form in terms of Stirling numbers and $r$-Stirling numbers of the second kind. \\

Recall that for integers $n$ and $k$ the Stirling numbers of the second kind, denoted by $\braces{n}{k}$, are the coefficients in the expansion
$$
x^n = \sum_{k=0}^n \braces nk (x)_k,
$$
where $(x)_k$ is the falling factorial defined by $(x)_{0}=1,(x)_{n}=x (x-1)\cdots (x-n+1)$ for $n>0$. We will often make use of the property
$$
\braces nk = 0, \text{if $k>n$.}
$$
The exponential generating function of $\braces{n}{k}$ equals
$$
\sum_{n\geq k} \braces{n}{k} \frac{z^{n}}{n!}=\frac{1}{k!} \left( e^{z}-1\right)^{k}.
$$
For any positive integer $r$, the $r$-Stirling numbers of the second kind $\braces{n+r}{k+r}_r$
are defined by the exponential generating function 
$$
\sum_{n\geq k} \braces{n+r}{k+r}_r \frac{z^{n}}{n!}=\frac{1}{k!} e^{rz} \left( e^{z}-1\right)^{k}.
$$
Details about these special numbers can be found in Laissaoui and Rahmani~\cite{laissaoui17} and references therein.

\begin{lemma}\label{stirling2}
If $k$, $r$ and $v$ are non-negative integers, then
\begin{align}
\left. \frac{d^r}{dx^r}\left( 1 - e^x \right)^k \right|_{x = 0} &= \sum_{p = 0}^k (- 1)^p \binom{k}{p} p^r = (- 1)^k k!\braces rk,\label{eq.r33yojo}\\
\left. \frac{d^r}{dx^r}\left( 1 - e^x \right)^k e^{vx} \right|_{x = 0} &= \sum_{p = 0}^k (- 1)^p \binom{k}{p} (v + p)^r 
= (- 1)^k k!\braces{r+v}{k+v}_v  \label{eq.vompp34},
\end{align}
where $\braces mn$ and $\braces{m+r}{n+r}_r$ are, respectively, Stirling numbers of the second kind and $r$-Stirling numbers.  
\end{lemma}
\begin{proof}
Since
\begin{equation*}
\left( {1 - e^x } \right)^k e^{vx} = \sum_{p = 0}^k (- 1)^p \binom{{k}}{p}e^{(v + p)x};
\end{equation*}
we have
\begin{equation*}
\frac{d^r}{dx^r}\left( \left( {1 - e^x } \right)^k e^{vx} \right) = \sum_{p = 0}^k (- 1)^p \binom{{k}}{p}\left(v + p \right)^r e^{(v + p)x};
\end{equation*}
and hence~\eqref{eq.vompp34}.
\end{proof}

\begin{theorem}\label{thm.id}
If $n$ and $r$ are non-negative integers and $t$ is not an integer, then
\begin{equation}\label{id}
\sum_{k = 0}^n (- 1)^k k^r \binom{t}{k} (H_t - H_{t - k}) = \sum_{k = 0}^r (-1)^k k! \braces{r+n-k}n_{n-k} \binom{n - t}{n - k} (H_{k - t} - H_{n - t}).
\end{equation}
\end{theorem}
\begin{proof}
Write $\exp (x)$ for $x$ in the known polynomial identity (variation on Gould~\cite[Identity (1.10),p.2]{Gould}):
\begin{equation}\label{gould110}
\sum_{k = 0}^n (- 1)^k \binom{{t}}{k} x^k = \sum_{k = 0}^n \binom{{n - t}}{k} \left( {1 - x} \right)^{n - k} x^k,
\end{equation}
and differentiate $r$ times with respect to $x$ to obtain
\begin{align*}
\sum_{k = 0}^n (- 1)^k k^r \binom{{t}}{k} e^{k x} &= \sum_{k = 0}^n \frac{d^r}{dx^r}\left( \left( 1 - e^x \right)^{n - k} e^{k x} \right)\binom{{n - t}}{k} \\
&= \sum_{k = 0}^n \frac{d^r}{dx^r} \left( \left( 1 - e^x \right)^k e^{(n-k)x} \right) \binom{{n - t}}{n-k}.
\end{align*}
Evaluation at $x=0$ yields
\begin{align*}
\sum_{k = 0}^n {( - 1)^k k^r \binom{{t}}{k} }  &= \left. {\sum_{k = 0}^n {\frac{{d^r }}{{dx^r }}\left( {\left( {1 - e^x } \right)^k e^{(n-k)x} } \right)\binom{{n - t}}{n-k}} } \right|_{x = 0}\\
 &= \sum_{k = 0}^n {\left. {\frac{{d^r }}{{dx^r }}\left( {\left( {1 - e^x } \right)^k e^{(n-k)x} } \right)} \right|_{x = 0} \binom{{n - t}}{n-k}}\\
 &= \sum_{k = 0}^r {(-1)^kk!\braces{r+n-k}n_{n-k} \binom{{n - t}}{n-k}},
\end{align*}
which gives~\eqref{id} upon differentiating with respect to $t$, since
\begin{equation*}
\frac{d}{{dt}}\binom{{t}}{k} = \binom{{t}}{k}\left( {H_t  - H_{t - k} } \right)
\end{equation*}
and
\begin{equation*}
\frac{d}{{dt}}\binom{{n - t}}{k} = \binom{{n - t}}{k}\left( {H_{n - t - k}  - H_{n - t} } \right).
\end{equation*}
\end{proof}

The four lowest explicit formulas for $r=0,1,2,3$ evaluations of~\eqref{id} are presented in Proposition~\ref{prop.m30jt31}.
\begin{proposition}\label{prop.m30jt31}
If $n$ is a non-negative integer and $t$ is not an integer, then
\begin{equation}
\sum_{k = 0}^n (- 1)^k \binom{{t}}{k}\left(H_t - H_{t - k} \right) = \binom{{n - t}}{n} \left(H_{ - t} - H_{n - t} \right),
\end{equation}
\begin{equation}
\sum_{k = 0}^n (- 1)^k k \binom{t}{k} \left(H_t - H_{t - k} \right) = -\binom{n - t}{n - 1}\left(H_{1 - t} - H_{n - t} \right) 
+ n\binom{n - t}{n} \left(H_{- t} - H_{n - t} \right),
\end{equation}
\begin{align}
\sum_{k = 0}^n (- 1)^k k^2 \binom{{t}}{k}\left( {H_t - H_{t - k} } \right) & = 2\binom{{n - t}}{{n - 2}}\left( {H_{2 - t} - H_{n - t} } \right)\nonumber\\
&\qquad - \left( {2n - 1} \right)\binom{{n - t}}{{n - 1}}\left( {H_{1 - t} - H_{n - t} } \right)\nonumber\\
&\qquad\quad + n^2 \binom{{n - t}}{n}\left( {H_{- t} - H_{n - t} } \right),
\end{align}
and
\begin{align}
\sum_{k = 0}^n (- 1)^k k^3 \binom{{t}}{k}\left( {H_t - H_{t - k} } \right) &=  - 6\binom{{n - t}}{{n - 3}}\left( {H_{3 - t} - H_{n - t} } \right)\nonumber\\
&\qquad + 6\left( {n - 1} \right)\binom{{n - t}}{{n - 2}}\left( {H_{2 - t} - H_{n - t} } \right)\nonumber\\
&\qquad\quad - \left( {3n^2 - 3n + 1} \right)\binom{{n - t}}{{n - 1}}\left( {H_{1 - t} - H_{n - t} } \right)\nonumber\\
&\qquad\qquad + n^3 \binom{{n - t}}{n}\left( {H_{ - t} - H_{n - t} } \right).
\end{align}
\end{proposition}

Theorem \ref{thm.id} also provides a closed form for the sums $S_r^O(n)$.

\begin{proposition}
If $n$ and $r$ are non-negative integers, then
\begin{align}
S_r^O(n) &= \sum_{k=0}^n 2^{-2k} k^r \binom{2k}{k} O_k \nonumber \\
&= 2^{-2n} \sum_{k = 0}^r (-1)^k k!\braces{r+n-k}n_{n-k} 2^{2k} \binom{2n + 1}{2k + 1} \binom{2(n - k)}{n - k} \binom{n}{k}^{- 1} (O_{n + 1} - O_{k + 1}).
\end{align}
In particular,
\begin{align*}
S_0^O (n) &= \sum_{k = 0}^n 2^{- 2k} \binom{2k}{k} O_k = (2n + 1) \binom{2n}{n} 2^{- 2n} \left( O_{n + 1} - 1 \right), \\
S_1^O (n) &= \sum_{k = 0}^n 2^{- 2k} k \binom{2k}{k} O_k = \frac{n (2n + 1)}{3} \binom{2n}{n} 2^{- 2n} \left( O_{n + 1} - \frac{1}{3} \right).
\end{align*}
\end{proposition}
\begin{proof}
Set $t=-1/2$ in~\eqref{id}. Use
\begin{equation*}
\binom{-1/2}{p} = (-1)^p \binom{2p}{p} 2^{-2p}
\end{equation*}
and
\begin{equation*}
\binom{p + 1/2}{q} = 2^{- 2q} \binom{2p + 1}{2q} \binom{2q}{q} \binom{p}{q}^{- 1},
\end{equation*}
and the fact that
\begin{equation*}
H_{h-1/2} = 2O_h -2\ln 2.
\end{equation*}
\end{proof}

\begin{proposition}
If $n$ and $r$ are non-negative integers, then
\begin{align}
&\sum_{k = 1}^n \frac{2^{- 2k}}{2k - 1} k^r \binom{2k}{k} (1 - O_{k - 1})\\
&\qquad = 2^{- 2n} \binom{2n}{n} \sum_{k = 0}^r (-1)^kk!\braces{r+n-k}n_{n-k}\binom{n}{k} 2^{2k} \binom{2k}{k}^{- 1} (O_n - O_k).
\end{align}
In particular,
\begin{equation}\label{Riordan2}
\sum_{k = 1}^n \frac{2^{- 2k}}{2k - 1} \binom{2k}{k} (1 - O_{k - 1}) = 2^{- 2n} \binom{2n}{n} O_n.
\end{equation}
\end{proposition}
\begin{proof}
Set $t=1/2$ in~\eqref{id}. Use
\begin{equation*}
\binom{{1/2}}{k} = \frac{(- 1)^{k + 1} 2^{- 2k}}{2k - 1} \binom{2k}{k}
\end{equation*}
and
\begin{equation*}
\binom{n - 1/2}{n - k} = 2^{2k - 2n} \binom{2n}{n} \binom{2k}{k}^{- 1} \binom{n}{k},
\end{equation*}
and the fact that $H_{1/2-k}=2O_{k-1}-2\ln 2$.
\end{proof}

Combining \eqref{Riordan2} with \eqref{Riordan1} we get for $n\geq 1$:
\begin{corollary}
If $n$ is a non-negative integer, then
\begin{equation}
\sum_{k=1}^n \frac{2^{-2k}}{2k-1} \binom{2k}{k} O_{k-1} = 1 - 2^{-2n} \binom{2n}{n} (O_n + 1).
\end{equation}
\end{corollary}

\begin{theorem}
If $n$ and $r$ are non-negative integers and $t$ is not an integer, then
\begin{align}\label{eq.zha71dp}
\sum_{k = 0}^n {( - 1)^k \binom{{t}}{{n - k}}k^r \left( {H_t  - H_{t - n + k} } \right)}  = ( - 1)^n \sum_{k = 0}^r {k!\braces{r}{k}\binom{{n - t}}{{n - k}}\left( {H_{k - t}  - H_{n - t} } \right)} .
\end{align}

\end{theorem}

\begin{proof}
Write $1/x$ for $x$ in~\eqref{gould110} to obtain
\begin{equation}\label{eq.zzq9wz8}
\sum_{k = 0}^n {( - 1)^k \binom{{t}}{{n - k}}x^k }  = \sum_{k = 0}^n {( - 1)^{n - k} \binom{{n - t}}{{n - k}}\left( {1 - x} \right)^k } .
\end{equation}
Now write $\exp x$ for $x$, differentiate $r$ times with respect to $x$ and evaluate at $x=0$ using Lemma~\ref{stirling2}. Finally differentiate the resulting expression with respect to $t$ using
\begin{equation}
\frac{d}{{dt}}\binom{{t}}{{n - k}} = \binom{{t}}{{n - k}}\left( {H_t  - H_{t - n + k} } \right)
\end{equation}
and
\begin{equation}
\frac{d}{{dt}}\binom{{n - t}}{{n - k}} = \binom{{n - t}}{{n - k}}\left( {H_{k - t}  - H_{n - t} } \right).
\end{equation}
\end{proof}

\begin{proposition}
If $n$ and $r$ are non-negative integers, then
\begin{align}
&\sum_{k = 0}^n {2^{2k} \binom{{2\left( {n - k} \right)}}{{n - k}}k^r O_{n - k} }\nonumber\\
&\qquad  = \left( {2n + 1} \right)\binom{{2n}}{n}\sum_{k = 0}^r {\frac{{2^{2k} }}{{2k + 1}}k!\braces{r}{k}\binom{{n}}{k}\binom{{2k}}{k}^{ - 1} \left( {O_{n + 1}  - O_{k + 1} } \right)}. 
\end{align}
\end{proposition}

\begin{proof}
Set $t=-1/2$ in~\eqref{eq.zha71dp}. Use
\begin{align*}
\binom{{ - 1/2}}{{n - k}} = ( - 1)^{n - k} \binom{{2\left( {n - k} \right)}}{{n - k}}2^{ - 2\left( {n - k} \right)} ,\\
\binom{{n + 1/2}}{{n - k}} = \frac{{2n + 1}}{{2k + 1}}\binom{{2n}}{n}\binom{{2k}}{k}^{ - 1} 2^{ - 2\left( {n - k} \right)} \binom{{n}}{k},\\
H_{ - 1/2}  - H_{ - 1/2 - n + k}  =  - 2O_{n - k},
\end{align*}
and
\begin{equation*}
H_{k + 12}  - H_{n + 1/2}  = 2\left( {O_{k + 1}  - O_{n + 1} } \right).
\end{equation*}

\end{proof}

\begin{proposition}
If $n$ and $r$ are non-negative integers, then
\begin{align}
\sum_{k = 0}^{n - 1} {\frac{{2^{2k} }}{{2\left( {n - k} \right) - 1}}\binom{{2\left( {n - k} \right)}}{{n - k}}k^r \left( {1 - O_{n - k - 1} } \right)}  = \binom{{2n}}{n}\sum_{k = 0}^r {2^{2k} \binom{{2k}}{k}^{ - 1} k!\braces{r}{k}\binom{{n}}{k}\left( {O_n  - O_k } \right)} .
\end{align}
\end{proposition}

\begin{proof}
Set $t=1/2$ in~\eqref{eq.zha71dp}. Use
\begin{align*}
\binom{{1/2}}{{n - k}} = \frac{{( - 1)^{n - k + 1} 2^{ - 2\left( {n - k} \right)} }}{{2\left( {n - k} \right) - 1}}\binom{{2\left( {n - k} \right)}}{{n - k}},\\
\binom{n - 1/2}{n - k} = 2^{2k - 2n} \binom{2n}{n} \binom{2k}{k}^{- 1} \binom{n}{k},\\
H_{1/2}  - H_{1/2 - n + k}  = 2\left( {1 - O_{n - k - 1} } \right),
\end{align*}
and
\begin{equation*}
H_{n - 1/2}  - H_{k - 1/2}  = 2\left( {O_n  - O_k } \right).
\end{equation*}

\end{proof}

\begin{theorem}
If $n$ and $r$ are non-negative integers and $t$ is not an integer, then
\begin{align}\label{eq.z15r0vl}
\sum_{k = 0}^n {( - 1)^{n-k} \binom{{n - t}}{{n - k}}k^r \left( {H_{k - t}  - H_{n - t} } \right)}  = \sum_{k = 0}^r {k!\braces{r}{k}\binom{{t}}{{n - k}}\left( {H_t  - H_{t - n + k} } \right)} .
\end{align}

\end{theorem}

\begin{proof}
A consequence of the
\begin{equation*}
\binom{{t}}{{n - k}} \leftrightarrow ( - 1)^n \binom{{n - t}}{{n - k}}
\end{equation*}
symmetry of~\eqref{eq.zzq9wz8}.
\end{proof}

\begin{proposition}
If $n$ and $r$ are non-negative integers, then
\begin{align}
&\sum_{k = 0}^n {\frac{{( - 1)^k  }}{{2k + 1}}\binom{{n}}{k}2^{2k}\binom{{2k}}{k}^{ - 1} k^r \left( {O_{n + 1}  - O_{k + 1} } \right)}\nonumber\\
&\qquad  = \frac1{2n + 1}\binom{2n}n^{-1}\sum_{k = 0}^r {( - 1)^k k!\braces{r}{k}2^{2k} \binom{{2\left( {n - k} \right)}}{{n - k}}O_{n - k} } .
\end{align}
\end{proposition}

\begin{proof}
Set $t=-1/2$ in~\eqref{eq.z15r0vl}.
\end{proof}
In particular,
\begin{equation}\label{eq.k194mhu}
\sum_{k = 0}^n {\frac{{( - 1)^k  }}{{2k + 1}}\binom{{n}}{k}2^{2k}\binom{{2k}}{k}^{ - 1} O_{k+1}}=\frac1{(2n+1)^2},
\end{equation}
since
\begin{equation*}
\sum_{k = 0}^n {\frac{{( - 1)^k 2^{2k} }}{{2k + 1}}\binom{{n}}{k}\binom{{2k}}{k}^{ - 1} }  = \frac{1}{{2n + 1}}.
\end{equation*}
From~\eqref{eq.k194mhu}, we also have
\begin{equation}
\sum_{k = 0}^n {( - 1)^k \binom{{n}}{k}\frac{1}{{\left( {2k + 1} \right)^2 }}}  = \frac{{2^{2n} }}{{2n + 1}}\binom{{2n}}{n}^{ - 1} O_{n + 1} .
\end{equation}
\begin{proposition}
If $n$ and $r$ are non-negative integers, then
\begin{align}
&\sum_{k = 0}^n {( - 1)^k \binom{{n}}{k}2^{2k} \binom{{2k}}{k}^{ - 1} k^r \left( {O_n  - O_k } \right)}\nonumber\\
&\qquad  = \binom{{2n}}{n}^{ - 1} \sum_{k = 0}^r {\frac{{( - 1)^k }}{{2\left( {n - k} \right) - 1}}k!\braces{r}{k}2^{2k} \binom{{2\left( {n - k} \right)}}{{n - k}}\left( {1 - O_{n - k - 1} } \right)} .
\end{align}
\end{proposition}

\begin{proof}
Set $t=1/2$ in~\eqref{eq.z15r0vl}.
\end{proof}
In particular,
\begin{equation*}
\sum_{k = 0}^n {( - 1)^k \binom{{n}}{k}2^{2k} \binom{{2k}}{k}^{ - 1}\left( {O_n  - O_k } \right)}=\frac{1-O_{n-1}}{2n - 1},
\end{equation*}
which simplifies to give
\begin{equation}\label{eq.gbx6l42}
\sum_{k = 0}^n {( - 1)^k \binom{{n}}{k}2^{2k} \binom{{2k}}{k}^{ - 1}O_k}=-\frac{2n}{(2n - 1)^2},
\end{equation}
in view of~\eqref{eq.fw0v51n}; and which by the binomial transformation also gives
\begin{equation}
\sum_{k = 0}^n {( - 1)^k \binom{{n}}{k}\frac{{2k}}{{\left( {2k - 1} \right)^2 }}}  =  - 2^{2n} \binom{{2n}}{n}^{ - 1} O_n .
\end{equation}



\end{document}